\pdfoutput=1
\RequirePackage[english]{babel}
\documentclass{amsart}

\usepackage{amssymb,amscd,amsthm}

\title{Discrete subgroups of locally definable groups}

\author[A.~Berarducci]{Alessandro Berarducci}
\address{Universit\`a di Pisa, Dipartimento di Matematica, Largo Bruno Pontecorvo 5, 56127 Pisa, Italy}
\email{berardu@dm.unipi.it}
\thanks{Partially supported by PRIN 2009WY32E8\_003: O-minimalit\`a, teoria degli insiemi, metodi e modelli non standard e applicazioni}

\author[M.~Edmundo]{M\'ario Edmundo}
\address{ Universidade Aberta and CMAF Universidade de Lisboa\\
Av. Prof. Gama Pinto 2\\
1649-003 Lisboa, Portugal}
\email{edmundo@cii.fc.ul.pt}
\thanks{Partially supported by Funda\c{c}\~ao para a Ci\^encia e a Tecnologia PEst OE/MAT/UI0209/2011}

\author[M.~Mamino]{Marcello Mamino}
\address{CMAF Universidade de Lisboa\\
Av. Prof. Gama Pinto 2\\
1649-003 Lisboa, Portugal}
\email{mamino@ptmat.fc.ul.pt}
\thanks{Partially supported by Funda\c{c}\~ao para a Ci\^encia e a Tecnologia grant SFRH/BPD/73859/2010}

\date{September 4, 2012}

\DeclareMathOperator{\R}{\mathbb R}

\DeclareMathOperator{\NN}{\mathbb N}
\DeclareMathOperator{\Q}{\mathbb Q}
\DeclareMathOperator{\Z}{\mathbb Z}

\DeclareMathOperator{\rank}{rank}

\newcommand{\nin}{\not\in}

\newcommand{\rh}{\operatorname{ch}}

\theoremstyle{plain}
\newtheorem{theorem}{Theorem}

\newtheorem*{Theorem}{Theorem}
\newtheorem{lemma}[theorem]{Lemma}

\newtheorem{proposition}[theorem]{Proposition}
\newtheorem{corollary}[theorem]{Corollary}
\newtheorem{conjecture}[theorem]{Conjecture}

\newtheorem{question}[theorem]{Question}
\newtheorem{fact}[theorem]{Fact}

\theoremstyle{definition}
\newtheorem{remark}[theorem]{Remark}
\newtheorem{definition}[theorem]{Definition}
\newtheorem{notation}[theorem]{Notation}
\newtheorem{discussion}[theorem]{Discussion}
\newtheorem{example}[theorem]{Example}
\newtheorem{exercise}[theorem]{Exercise}

\numberwithin{theorem}{section}

\newcommand{\bt}{\begin{theorem}}
\newcommand{\et}{\end{theorem}}

\newcommand{\bl}{\begin{lemma}}
\newcommand{\el}{\end{lemma}}

\newcommand{\bfa}{\begin{fact}}
\newcommand{\efa}{\end{fact}}

\newcommand{\bexa}{\begin{example}}
\newcommand{\eexa}{\end{example}}

\newcommand{\bexe}{\begin{exercise}}
\newcommand{\eexe}{\end{exercise}}

\newcommand{\bprop}{\begin{proposition}}
\newcommand{\eprop}{\end{proposition}}

\newcommand{\bp}{\begin{proof}}
\newcommand{\ep}{\end{proof}}

\newcommand{\bc}{\begin{corollary}}
\newcommand{\ec}{\end{corollary}}

\newcommand{\bd}{\begin{definition}}
\newcommand{\ed}{\end{definition}}

\newcommand{\br}{\begin{remark}}
\newcommand{\er}{\end{remark}}

\def\N{\mathbb{N}}
\def\st{\;|\;}

\newenvironment{acknowledgements}{{\medskip \noindent \bf Acknowledgements.}}{}

\def\url#1{\href{#1}{url\nobreakdash---\texttt{#1}}}

\begin{document} 

\begin{abstract}
We work in the category of locally definable groups in an o-minimal expansion of a field.  Eleftheriou and Peterzil conjectured that every definably generated abelian connected group $G$ in this category, is a cover of a definable group. We prove that this is the case under a natural convexity assumption inspired by the same authors, which in fact gives a necessary and sufficient condition. The proof is based on the study of the zero-dimensional compatible subgroups of $G$. Given a locally definable connected group $G$ (not necessarily definably generated), we prove that the $n$-torsion subgroup of $G$ is finite and that every zero-dimensional compatible subgroup of $G$ has finite rank. Under a convexity hypothesis we show that every zero-dimensional compatible subgroup of $G$ is finitely generated. 
\end{abstract}

\maketitle

\section{Introduction} 
The notion of o-minimal structure provides a general framework for the study of various ``tameness phenomena'' typical of semi-algebraic, sub-analytic, and similar categories of sets (see \cite{vdDries1998}). A basic example of an o-minimal structure is given by the ordered field of real numbers $(\R, <, +,\cdot)$ and the definable sets in this structure are exactly the semi-algebraic sets, namely a subset of $\R^n$ is semi-algebraic if and only if it is definable in the field of real numbers. 
There are also interesting expansions of the real field which are o-minimal, such as $(\R,<,+,\cdot,\exp)$, where $\exp$ is the real exponential function.

Any structure elementary equivalent to an o-minimal structure is o-minimal. Thus in particular every real closed field $(M,<,+,\cdot)$ is o-minimal. In order to be able to apply the usual tools of model theory (types, saturated models, etc.) it is convenient to consider, together with a given o-minimal structure, also all the structures that are elementary equivalent to it. For a general reference to model theory the reader may consult \cite{Tent2012}. 

In this paper we consider groups which are ``locally definable'' in an o-minimal expansion $M = (M,<,+,\cdot, \ldots)$ of an ordered field (necessarily real closed). 
The study of definable groups has been a main theme of research in the model theory of o-minimal structures. 
More recently various authors considered the larger category of those groups which are {\em locally definable}, namely the domain of the group and the graph of the group operation are given by countable unions of definable sets. One of the motivations for working in this larger category is that the universal cover of a definable group is locally definable \cite{Edmundo2007}. Unlike the case of definable groups, in general one cannot expect a ``tame'' behaviour for all locally definable groups. For instance every countable group is obviously locally definable. However under some natural additional assumptions, such as connectedness (in the sense of \cite{Baro2010}), the known examples of locally definable groups seem to exhibit a tame behavior (when $M$ is an expansion of the reals these groups are connected real Lie groups).  
In particular it is natural to conjecture that a locally definable abelian connected group $G$ behaves in many respects like a finite product of (non-standard) copies of $(\R,+)$ and $\R/\Z$, and in particular it is divisible and has an $n$-torsion subgroup isomorphic to $(\Z/n\Z)^s$ for some $s\leq \dim(G)$. 

This is indeed true in the definable case (with $s = \dim (G)$ if $G$ is definably compact \cite{Edmundo2004}), and remains true for those abelian connected locally definable groups $G$ which are {\em covers} of definable groups (namely there is a locally definable surjective homomorphism from $G$ to a definable group $H$ of the same dimension). 

Let us observe that a locally definable cover of a definable group is always {\em definably generated} (see \cite{Eleftheriou2012a}), namely it is generated as an abstract group by a definable subset. For an example of a locally definable group which is not definably generated one can take a non-archimedean real closed field $(M,<,+,\cdot)$ and a convex subgroup $G$ of the ordered additive group $(M,<,+)$ of the form $\bigcup_{i\in \NN} (-a_i,a_i)$ where $na_i < a_{i+1}$ for all $n\in \N$ \cite{Eleftheriou2012a,Eleftheriou2012c}. Eleftheriou and Peterzil made the following conjecture: 

\begin{conjecture}\cite{Eleftheriou2012a,Eleftheriou2012b} \label{cong-cover} Let $G$ be a definably generated abelian connected group. Then $G$ is a cover of a definable group.  \end{conjecture}

A solution of the conjecture would reduce many questions about locally definable abelian connected groups to similar questions about definable groups, even without assuming that the group is definably generated. This depends on the fact that every locally definable group $G$ is a directed union of definably generated subgroups (which can be taken to be connected if $G$ is connected). For instance a positive solution of Conjecture \ref{cong-cover} would yield a positive solution to the following conjecture:

\begin{conjecture}\cite{Edmundo2003a,Edmundo2005,Eleftheriou2012a}
 Every locally definable connected abelian group is divisible.
 \end{conjecture} 
    
In \cite{Eleftheriou2012b} Eleftheriou and Peterzil proved Conjecture \ref{cong-cover} for subgroups of $(M^n,+)$. They also proved that Conjecture \ref{cong-cover} is equivalent to the following: 

\begin{conjecture}\cite[Conjecture B] {Eleftheriou2012b} \label{vdim} Let $G$ be a definably generated abelian connected group. 
\begin{enumerate}
\item There is a maximal $k\in \N$ such that $G$ contains a compatible subgroup $\Gamma$ isomorphic to $\Z^k$. 
\item If $G$ is not definable, $k\geq 1$. 
\end{enumerate}
\end{conjecture} 

If Conjecture \ref{vdim} holds and we take $\Gamma$ as in the conjecture, then $G/\Gamma$ is definable, and $G$ covers it. 

We are now ready to discuss the results of this paper. In Section \ref{rank} we prove part (1) of Conjecture \ref{vdim}. In Section \ref{sec-convexity} we prove part (2) under a convexity assumption suggested by the work of Eleftheriou and Peterzil.  As a result of this analysis we show:

\begin{Theorem}[see Theorem \ref{convex-main}]
A definably generated connected abelian group $G$ is a cover of a definable group if and only if for every definable set $X\subseteq G$ there is a definable set $Y\subseteq G$ which contains the convex hull of $X$ (in the sense of Definition \ref{convex}).
\end{Theorem}

We also consider locally definable connected abelian groups which are not necessarily definably generated. 
Let $G$ be a locally definable abelian connected group. We show that $G$
has a finite $n$-torsion subgroup~$G[n]$ (Corollary \ref{torsion}) and that every zero-dimensional compatible subgroup $\Gamma$ of $G$ has finite rank bounded by the dimension of $G$. 
Under an additional convexity assumption we show that $\Gamma$ is finitely generated (Corollary~\ref{fgdisc}). Since the fundamental group $\pi_1(G)$ of $G$ is isomorphic to a zero-dimensional compatible subgroup of the universal cover $U$ of $G$, the above results can be applied to bound the rank of $\pi_1(G)$ (Corollary \ref{pi-one}) or to prove (under the appropriate convexity hypothesis) that $\pi_1(G)$ is finitely generated (Theorem \ref{fingen}). The actual proof however goes the other way around: first we give bounds on the fundamental groups, and then we deduce bounds on the zero-dimensional compatible subgroups. 

The problem of the finiteness of the $n$-torsion subgroup
 was considered in the unpublished note \cite{Edmundo2003a}
by the second author, but the proof of finiteness of $G[n]$ contained
therein had some gaps (and assumed divisibility). Among the motivations of this paper there
is a desire to provide a correct proof. 

In Section \ref{examples} we give an example of a locally definable
H-space whose fundamental group is $\Q$. This may explain some of the
difficulties in showing that $\pi_1(G)$ is finitely generated using only
homotopy information,  thus giving some indirect justification for the introduction of the convexity assumption. In the same section we give an example of a locally definable connected group which has a generic definable subset, but does not cover a definable group. This shows that the results of \cite{Eleftheriou2012a} do not extend to non-commutative groups. 

The reader may gain some insight on the problems considered in this paper by first taking a look at Question \ref{ques} at the end of the paper. 

Sections \ref{ld-groups} and \ref{covers} contain some definitions and background results. 

\begin{acknowledgements} The main results of this paper have been presented on Feb. 2, 2012 at the Logic Seminar of the Mathematical Institute in Oxford. The first author thanks Jonathan Pila for the kind invitation. We also thank Margarita Otero, Pantelis Eleftheriou, Kobi Peterzil and the anonymous referee for their comments. \end{acknowledgements}

\section{Locally definable groups} \label{ld-groups}

We work in an o-minimal expansion~$M$ of a field. A subset of $M^n$ 
is definable if it is definable in~$M$ with parameters. 
There are several different definitions and variants of locally definable sets and groups in the literature.  
The definition is simpler if $M$ is assumed to be $\aleph_1$-saturated, so let us momentarily assume that this is the case. 
For us a~\textit{locally definable set} is a countable union $\bigcup_{i\in \N} X_i$ 
of definable subsets of $M^n$ for some fixed $n$, and a \textit{locally definable function} $f:X \to Y$ between locally definable sets $X = \bigcup_{i\in \N}X_i$ and $Y = \bigcup_{j\in \N} Y_j$ is a function whose restriction to each definable subset of its domain $X$ is definable. Given a locally definable set $X = \bigcup_{i\in \N}X_i$ we can always assume that the union is {\em directed}, namely for every $i,j\in \N$ there is $k\in \N$ with $X_i\cup X_j \subseteq X_k$ (if not, we reduce to the directed case considering the unions of the finite subfamilies). By $\aleph_1$-saturation it then follows that every definable subset of $\bigcup_{i\in \N} X_i$ is contained in some $X_i$ (or in a finite union of the $X_i$'s if the union is not directed). 
In the non-saturated case the corresponding property fails, unless we make the convention that a locally definable set $X$ comes equipped with a given presentation as a directed countable union $\bigcup_{i\in \N}X_i$ of definable sets and that by a {\em definable subset} of $X$ we really mean a definable subset of some $X_i$. Consistently with this convention we stipulate that $f: X \to Y$ is locally definable if and only if for each $i\in \N$ the restriction of $f$ to $X_i$ has image contained in some $Y_j$ and it is definable as a function from $X_i$ to $Y_j$. For instance if $M=\R$ we have $M = \bigcup_{n\in \N} [-n,n]$ but the identity map $f: \R \to \bigcup_{n\in \N} [-n,n]$ is not locally definable since the image through $f$ of the definable set $\R$ is not contained in $[-n,n]$ for any $n\in \N$. 
A~\textit{locally definable group} is a locally definable set $G= \bigcup_{i\in \N} X_i$ equipped with a locally definable group operation $\mu: G\times G \to G$ and group inverse, where we endow $G\times G$ with the presentation $G\times G = \bigcup_{i,j\in \N}X_i \times X_j$ as a locally definable set. 

Given a definable set $X$ in $M$ and an elementary extension $M'\succ M$, we denote by $X(M')$ the set defined in $M'$ by the same formula defining $X$ (so in particular $X = X(M)$). It is easy to see that $X(M')$ does not depend on the choice of the definining formula for $X$. Similarly, given a locally definable set $X = \bigcup_{i\in \N}X_i$ we denote by $X(M')$ the locally definable set $\bigcup_{n\in \N}X(M')$. In the non-saturated case this depends on the presentation of $X$ as a countable union of definable sets. 

Every definable group has a unique group topology making it into a
definable manifold over~$M$ \cite{Pillay1988}. This result
was generalized to locally definable groups in \cite[Proposition~2.2]{Peterzil2000a}, \cite[Theorem~2.3]{Edmundo2006} and \cite[Theorem 3.9]{Baro2010}. By the latter reference every locally definable group admits a group topology making it into a~\textit{locally definable space} with a countable atlas. We call the resulting group topology
the~\textit{t-topology}. When $M$ is an o-minimal expansion of the reals we get a real Lie group. 
For the definition and a recent systematic treatment of
locally definable spaces  we refer to~\cite{Baro2010}. Roughly speaking locally definable spaces stand to locally definable sets as abstract manifolds (given by atlases) stand to submanifolds of $\R^n$. 
In~\cite{Baro2010} a well behaved subclass of the locally definable spaces is
discussed: the paracompact ones. They admit the following
characterization: a locally definable space~$X$ is~\textit{paracompact} 
if the closure of each definable subspace of~$X$ is definable~\cite[Fact 2.7]{Baro2010}. 
A locally definable group~$G$, with the t-topology, is always 
paracompact \cite[Theorem 3.10]{Baro2010}. Thus we can make free use of the results of~\cite{Baro2010} on locally
definable paracompact spaces when dealing with locally definable groups. 
Finally, let us recall the following definition: a subset~$X$ of a locally
definable group (or space)~$G$ is~\textit{compatible} if its intersection with
any definable subset of~$G$ is definable.
By~\cite[Lemma~3.3 and Theorem~4.2]{Edmundo2006} a normal
subgroup~$A\vartriangleleft B$ of a locally definable group~$B$ is
compatible if and only if it is the kernel of a locally definable
surjective homomorphism~$f\colon B \to C$ between locally definable
groups. So the quotients $B/A$, with~$A$ compatible in~$B$, exist in the
category of locally definable groups and they are unique up to locally
definable isomorphisms. Let us also recall that a locally definable space is~\textit{connected}
if it is not the union of two non-trivial clopen compatible subsets \cite{Baro2010}. This is equivalent to the condition that
every two points can be joined by a definable path. It is easy to see that a locally definable group is connected (in the t-topology) if and only if it has no compatible subgroups of index $\leq \aleph_0$.

\section{Covers}
\label{covers}
We recall a few facts from the theory of covering spaces in the locally definable category, as developed in \cite{Edmundo2005,Edmundo2007} (see \cite{Edmundo2004} for the definable case). Given two locally definable connected groups $U$ and $G$, a surjective locally definable homomorphism $f: U \to G$ is a {\em covering homomorphism} if and only if its kernel is zero-dimensional, or equivalently if $\dim(U)=\dim(G)$ \cite[Theorem 3.6]{Edmundo2005}. Given such a covering $f:U\to G$, there is an induced injective homomorphism $f_*: \pi_1(U)\to \pi_1(G)$ on the o-minimal fundamental groups (see~\cite[Proposition 4.6]{Edmundo2005} or \cite[Proposition~6.12]{Baro2010}). We say that $U$ is {\em simply connected} if $\pi_1(U) = 0$. 
By \cite[Prop. 3.4 and 3.12]{Edmundo2005} we have: $$\ker(f) \cong \pi_1(G)/f_*(\pi_1(U)).$$
In particular if $U$ is the {\em universal cover} of $G$ (as in \cite{Edmundo2007}), then $\pi_1(U) = 0$ and $\ker(f) \cong \pi_1(G)$. So $\pi_1(G)$ is isomorphic to a zero-dimensional subgroup of the universal cover of $G$. Another way of saying the same thing is $\Lambda \cong \pi_1(U/\Lambda)$, provided $\pi_1(U) = 0$ and $\Lambda$ is a zero-dimensional compatible normal subgroup of $U$. More generally if $\Gamma$ is a zero-dimensional compatible normal subgroup of a connected locally definable group $G$, then $\Gamma$ is a quotient of $\pi_1(G/\Gamma)$ (consider the covering $G\to G/\Gamma$). So there is a strong connection between compatible zero-dimensional subgroups and fundamental groups. 

\section{Discrete compatible subgroups} 
\label{rank} 

In this section we prove the finiteness of the the $n$-torsion subgroup $G[n]$ of a locally definable connected group $G$. We also prove the finiteness of the rank of any zero-dimensional compatible subgroup of $G$. We will make use of homological techniques.
Homology and homotopy in the locally definable category has been studied
by various authors: see~\cite{Baro2010} for some bibliography and recent
results.  Given a locally definable space~$X$, we let  $S_*(X)$ denote 
the o-minimal singular chain complex of~$X$, and $H_*(X;R)$ the corresponding
graded homology group with coefficients in~$R$.

\bt \label{hone} Let $G$ be a locally definable abelian connected group.
Then for any perfect field~$R$, the $R$-vector space~$H_1(G;R)$ has finite
dimension~$\leq \dim(G)$. In particular:
\begin{enumerate}
\item the $\Q$-vector space $H_1(G;\Q)$ has dimension~$\leq \dim(G)$,
\item the $\Z/p\Z$-vector space~$H_1(G;\Z/p\Z)$ has dimension~$\leq \dim(G)$.
\end{enumerate}
\et

The proof is based on the theory of Hopf-algebras (see \cite{Dold1995}). 
The same tool has been used in~\cite{Edmundo2004} for the study of definable groups but working in cohomology rather than homology. We have chosen to work in homology since the homological Kunneth formula does not require the assumption that the homology is of finite type (unlike the cohomological version). Note that in the locally definable category the finite type assumption is not granted {\em a priori}. Indeed it is easy to construct locally definable spaces which do not satisfy it. The Kunneth formula in homology is used in the proof of the following fact. 

\begin{fact} \label{hopf}
Let $R$ be a field and let $G$ be a locally definable abelian connected group. 
The group multiplication~$\mu\colon G \times G \to G$ induces on~$H_*(G;R)$ the structure of a commutative (in the graded sense) $R$-algebra, which is in fact a connected Hopf-algebra over $R$ in the sense of \cite{Dold1995}.
\end{fact}

\bp Given two locally definable spaces $X$ and $Y$ there is a natural chain homotopy equivalence
\[
S_*(X\times Y)
\rightleftarrows S_*(X)\otimes S_*(Y)
\]
This has been verified in the
definable case in \cite[Proposition~3.2]{Edmundo2004} and the  proof in the
locally definable case is identical. The Kunneth formula
\[
H_*(X\times Y;R)
\cong H_*(X;R) \otimes H_*(Y;R)
\]
then follows (see
\cite[ch.~\textsc{vi} Theorem~9.13]{Dold1995}). Identifying $H_*(G\times G;R)$ with $H_*(G;R) \otimes
H_*(G;R)$, the group multiplication $\mu\colon G \times G \to G$ induces
a multiplication $\mu_*\colon H_*(G;R) \otimes H_*(G;R) \to H_*(G;R)$
in homology making $H_*(G;R)$ into a Pontryagin's ring (see
\cite[ch.~\textsc{vii} Section~3]{Dold1995}). The diagonal map
$\Delta\colon  G \to G \times G$, defined by $\Delta(x) = (x,x)$, induces
a natural {\em co-multiplication} (or {\em diagonal}) $\psi \colon  H_*(G;R)
\to H_*(G;R) \otimes H_*(G;R)$ making the Pontryagin's ring $H_*(G;R)$
into a connected commutative Hopf-algebra (see
\cite[ch.~\textsc{vii}~10.10]{Dold1995}).
\ep

We now proceed to prove Theorem~\ref{hone}. The proof is easy when~$R=\Q$.
In fact, over a field of
characteristic zero, any connected commutative Hopf-algebra is free by the Hopf-Leray
theorem (see \cite[ch.~\textsc{vii} Proposition~10.16]{Dold1995}). 
For a contradiction take $\dim(G)+1$ elements in~$H_1(G;\Q)$ independent
over~$\Q$.
Since the Hopf $\Q$-algebra~$H_*(G;\Q)$ is free, their Pontryagin product
is a non-zero element of~$H_{\dim(G)+1}(G,\Q)$. This is absurd since $H_n(G,\Q)$
vanishes if~$n>\dim(G)$. In fact, the homology of a locally definable space is the direct limit of the homologies of its definable subsets by~\cite[Theorem~3.1]{Baro2010}, and a definable set of dimension~$d$ has trivial $H_n$ for~$n>d$ (see for instance~\cite[Lemma 3.1]{Edmundo2004}). To prove the theorem in characteristic different from zero we must employ a slightly more involved argument.

\bp[Proof of Theorem~\ref{hone}]
Let $H= H_*(G;R)$. By Fact~\ref{hopf}, $H$ is a connected commutative Hopf algebra. In
general the comultiplication sends a homology class $x\in H_q$ with
$q\geq 1$ to an element $\psi(x)$ of the form  $x \otimes 1 + 1 \otimes x
+ r$ where $r$ has the form $\Sigma_i a_i \otimes b_i$ with $|a_i| \geq 1$
and $|b_i| \geq 1$ (\cite[(10.10), p. 229]{Dold1995} or \cite[ch.~\textsc{iv}, \S
2, p. 173]{Mimura1978}). An element $x$ is primitive if $r = 0$, namely
$\psi(x) = x \otimes 1 + 1 \otimes x$.  Clearly each element of $H_1 :=
H_1(G;R)$ is primitive, since $\psi$ must preserve the degrees. Let $d= \dim(G)$. We must prove 
that $H_1$ has dimension~$\leq d$ over~$R$. For a contradiction suppose
that there are $d+1$ $R$-linearly independent elements $x_1, \ldots, x_{d+1}$
of~$H_1$.
Let $P \subseteq H$ be the subalgebra generated by $x_1, \ldots, x_{d+1}$. So in particular $P$ is generated by primitive elements of $H$. If $x$ is primitive, $\psi(x^h) = \Sigma_{i=0}^{h} { h \choose i} x^i \otimes x^{h-i}$. It follows that $\psi (P) \subseteq P \otimes P$, so $P$ is again a Hopf algebra, with comultiplication given by the restriction of $\psi$ to $P$. 
Note that since $P$ is generated by primitive elements, the
comultiplication is associative and commutative (see \cite[ch.~\textsc{vii}, Thm.
1.1, p. 366]{Mimura1978}), but we will not need this fact. What we need
instead is that $P$ is a connected Hopf algebra of finite type over a
perfect field with an associative and commutative {\em multiplication}. So
by the Hopf-Borel Theorem (see \cite[(8.11), p. 154]{Whitehead1978} or
\cite[Theorem 1.3, p. 366]{Mimura1978} $P$ is isomorphic, as a graded
algebra, to a tensor product of algebras $\bigotimes_{i=1}^n B_i$, where
each $B_i$ is monogenic, namely it is generated by a single element $y_i$
with $|y_i| > 0$. The isomorphism need not preserve the comultiplication,
so it is not in general an isomorphism of Hopf algebras. In any case the
product $y_1\cdot \ldots \cdot y_n$ is a non-zero element of
$\bigotimes_{i=1}^n B_i$ of degree $m = |y_1| + \ldots + |y_n| \geq n$. Since the isomorphism $P
\cong \bigotimes_ i B_i$ preserves the degrees, we obtain $P_m \neq 0$, so also
$H_m(G;R) \neq 0$. This implies $m \leq \dim(G)$ and since $n\leq m$, we obtain $n\leq \dim(G)$. 
On the other hand since each $B_i$ is monogenic ($i=1, \ldots, n$), the subspace of $B_i$ consisting of the elements of degree $1$ has dimension $\leq 1$, and therefore the subspace of $\bigotimes_{i=1}^n B_i$ consisting of the elements of degree $1$ has dimension $\leq n$. The same then holds for the isomorphic algebra $P$, which is a contradiction since $P$ has $d+1 \geq n$ independent elements of degree~$1$.
\ep 

Let us recall that
the {\em rank} of an abelian group~$\Gamma$ is the cardinality of a maximal
linearly independent (over~$\mathbb{Z}$) subset of~$\Gamma$; it coincides with
the dimension of~$\Gamma\otimes \Q$ as a $\Q$-vector space.

\bc \label{pi-one}
Let $G$ be a locally definable abelian connected group and let $R$ be a perfect field. Then the dimension of $\pi_1(G) \otimes R$ as a vector space over $R$ is $\leq \dim(G)$. In particular: 
\begin{enumerate}
 \item $\pi_1(G)$ has rank $\leq \dim(G)$.  
\item $\pi_1(G) \otimes \Z/p\Z$ has dimension $\leq \dim(G)$ over $\Z/p\Z$ for each prime $p$. 
\end{enumerate}
\ec 
\bp
Since $G$ is a locally definable group, $\pi_1(G)$ is abelian. Hence, by
the locally definable version of the Hurewicz theorem~\cite[Theorem~6.15]{Baro2010}, $\pi_1(G)$ is isomorphic to~$H_1(G;\Z)$. It follows that
$\pi_1(G) \otimes R$ is isomorphic to~$H_1(G;R)$. The statement now follows
from Theorem~\ref{hone}.
\ep 

Note that the above corollary does not imply that $\pi_1(G)$ is finitely generated, for instance it does not rule out the possibility that $\pi_1(G) \cong \Q$. 
The following result proves part~(1) of Conjecture~B
in~\cite{Eleftheriou2012b}. 

\bc\label{conj-b}
Let $G$ be a locally definable abelian connected group. Let $\Gamma <
G$ be a zero-dimensional compatible subgroup of~$G$. Then for every perfect field $R$ the dimension of $\Gamma \otimes R$ as a vector space over $R$ is~$\leq \dim(G)$. In particular: 
\begin{enumerate}
\item $\Gamma$ has rank $\leq \dim(G)$. 
\item $\Gamma \otimes \Z/p\Z$ has dimension $\leq \dim(G)$ over $\Z/p\Z$ for each prime $p$.  
\end{enumerate}
\ec
\bp
Let $\pi\colon U \to G$ be the universal cover of~$G$ and let $\Lambda =
\pi^{-1}(\Gamma) < U$. Then $\Lambda$ is a compatible zero-dimensional
subgroup of~$U$. By Corollary~\ref{pi-one}, $$\dim_R(\pi_1(U/\Lambda) \otimes R) \leq\dim(U/\Lambda) = \dim(G)$$
where the last equality follows from the fact that $\Lambda$ is zero-dimensional. 
Now since $U$ is simply connected, $\Lambda \cong \pi_1(U/\Lambda)$
by the theory of covers, and since $\Gamma$ is a homomorphic image
of~$\Lambda$, it follows that $\dim_R(\Gamma \otimes R) \leq \dim_R(\Lambda \otimes R) \leq \dim(G)$. 
\ep

We are now ready to prove the finiteness of the $n$-torsion subgroup $G[n]$ of a locally definable connected group $G$. A natural
approach would to consider the homomorphism $n\colon G \to G$, $x\mapsto
nx$. This has been proved to be effective in the definable case
\cite{Edmundo2004} and can still be used in the locally definable case to
show that $G[n]$ is a compatible zero-dimensional subgroup of $G$
\cite[Proposition~3.1]{Eleftheriou2012a}. However, to prove that the
homomorphism $n\colon G \to G$ is surjective, hence a covering, we would need
the divisibility of $G$, which in the locally definable case is still
conjectural. So we follow a different approach based on Corollary \ref{conj-b}(2) and an inductive argument. 

\bc\label{torsion}
Let $G$ be a locally definable abelian connected group of dimension $d$.
Then for each positive integer $n$ the $n$-torsion subgroup~$G[n]$ of~$G$ is finite, and it has at most $n^d$~elements.
\ec
\bp By  \cite[Proposition~3.1]{Eleftheriou2012a} $G[n]$ is a compatible subgroup of $G$. In particular, for $p$ prime, $G[p]$ is compatible and by Corollary \ref{conj-b}(2) we have $G[p] \cong {(\Z/p\Z)}^s$ for some $s\leq d$. The desired result now follows by elementary arguments in abstract group theory using only the fact that $G$ is abelian. Indeed note that if $p$ is prime, multiplication by $p$ gives a homomorphism $p: G[p^{k+1}] \to G[p^k]$ with kernel $G[p]$. So $|G[p^{k+1}]| \leq |G[p]| \cdot |G[p^k]|$ and by induction on $k$, $G[p^k]$ has at most 
$|G[p]|^k = {(p^s)}^k$ elements. This yields the desired result when $n$ is a power of a prime. The general case follows by considering the prime decomposition $n = p_1^{\alpha_1} \cdot \ldots \cdot p_l^{\alpha_l}$ of~$n$ and the isomorphism $G[n]\cong G[p_1^{\alpha_1}] \oplus \ldots \oplus G[p_l^{\alpha_l}]$.
\ep

If $G$ is divisible we can strengthen the conclusion of Corollary \ref{torsion} as follows. 

\br \label{div} Let $G$ be a locally definable connected divisible abelian group
of dimension $d$. Then for each prime $p$ and positive integer $k$, $G[p^k] \cong (\Z/p^k\Z)^s$ for some $s\leq d$ possibly depending on $p$ but not on $k$. 
\er
\bp 
Let $x_1,\dotsc, x_s\in G[p]$ be a
basis of~$G[p]$ as a vector space over~$\Z/p\Z$. By Corollary \ref{torsion} we have $s \leq d$. Since $G$ is divisible, there are~$y_1, \dotsc, y_s\in G$ such that $p^{k-1}y_i = x_i$ for~$i=1,\ldots,s$.
It is then easy to verify, by induction on~$k$, that $y_1, \dotsc, y_s$
have order~$p^k$, and $G[p^k]$ is a free $\Z/p^k\Z$-module generated by~$y_1, \ldots, y_s$. 
\ep

Remark \ref{div} does not settle the question whether there is some relation between $G[p]$ and $G[q]$ for different primes $p,q$. For instance, even assuming divisibility, we are not able to exclude the possibility that $G[6] \cong (\Z/3\Z)^2 \times \Z/2\Z$. However in Corollary \ref{same} we will show, under an additional convexity assumption, that $G[n] \cong (\Z/n\Z)^s$ for some $s\leq \dim(G)$ depending only on $G$ and not on $n$.

\section{Convexity}\label{sec-convexity}

In this section we prove the Conjectures of Eleftheriou and Peterzil mentioned in the introduction under a convexity assumption. Let us first recall the following definition. 

\bd \label{bounded}
Let $G$ be a locally definable group in a $\kappa$-saturated strongly $\kappa$-homogeneous o-minimal structure $M$ for some sufficiently big cardinal $\kappa$. 
\begin{enumerate}
\item A subset $X\subseteq G$ is {\em type-definable} if it is the intersection $<\kappa$ definable sets. 
\item A type-definable subgroup $H<G$ has {\em bounded index} if there are no new cosets of $H$ in $G$ in elementary extensions of $M$. Equivalently $[G:H]<\kappa$. 
\item If $H\lhd G$ has bounded index, we say that $X \subseteq G/H$ is open in the {\em logic topology} if its preimage in $G$ is the union of $<\kappa$ definable sets. 
\item If there is a smallest type-definable subgroup of $G$ of bounded index we call it $G^{00}$ and say that $G^{00}$ exists. 
\end{enumerate}
\ed

For $G$ definable, $G^{00}$ exists and $G/G^{00}$ is a real Lie group \cite{Berarducci2005}.  
For $G$ locally definable we have:

\bt \cite{Eleftheriou2012a} \label{G00} Let $G$ be an abelian, connected, definably generated group of dimension $d$. Then:
\begin{enumerate}
\item The subgroup $G^{00}$ exists if and only if $G$ covers a definable group;
\item If $G^{00}$ exists, then $G^{00}$ is divisible and $G/G^{00}$ is a Lie group isomorphic to $\R^k \times {(\R/\Z)}^{r}$ for some $k,r$ with $k+r \leq d$. 
\end{enumerate}
\et

Although we will not need it explicitly, let us also recall that, under the same hypothesis, in \cite{Eleftheriou2012a} it is also shown that $G^{00}$ exists if and only if $G$ has a generic definable subset, where a subset $X$ of $G$ is {\em generic} if for every definable $Y\subseteq G$ finitely many translates of $X$  cover $Y$. 

The other ingredient that we need is the following notion of convexity. 

\begin{definition}\label{convex}
Let $C$ be a subset of an abelian group~$G$. We say that $C$ is {\em convex} if for every 
$a,b\in C$ and $m,n\in \N$, not both null, $C$ contains every solution $x\in G$ of the equation 
$(m+n)x=ma+nb$. Note that if $G$ is divisible, there will be at least one solution. Given a subset $X\subseteq G$, the {\em convex hull} $\rh (X)$ of $X$ is the smallest convex set containing $X$. It can be equivalently defined as the set of all $x\in G$ such that $kx = a_1+ \ldots + a_k$ for some positive integer $k$ and some $a_1, \ldots, a_k \in X$ not necessarily distinct.  Note that the convex hull of a definable set need not be definable. We say that a locally definable abelian group has~{\em definably bounded convex hulls} if for all definable~$X\subseteq G$ there is a
definable~$Y\subseteq G$ containing $\rh(X)$.
\end{definition}

\def\repsum#1#2{\Sigma_{#1}{#2}}
\def\mult#1#2{{#1}{#2}}

\begin{notation} Let $X$ be a subset of an abelian group~$G$. We write $nX$ for $\left\{ nx\;|\;x\in X \right\}$ and $\repsum n X$ for $\{x_1+\ldots + x_n \st x_1,\ldots,x_n \in X\}$, where $+$ is the group operation. 
\end{notation}

\br \label{simple} If $G$ is divisible and torsion free, then $X$ is convex if and only if $nX = \repsum n X$ for every positive $n\in \N$. In this case $G$ is a $\Q$-vector space and convexity has the usual meaning (the $\Q$-segment between two points in $X$ is contained in $X$). If $G$ is only assumed to be divisible, then $X$ is convex if and only if we have {\em both} $nX = \repsum n X$ for every positive $n\in \N$ {\em and} $g+X = X$ for every torsion element $g\in G$. Example: the convex hull of the zero element of $G$ is the torsion subgroup of $G$. 
\er

We are now ready to state the main result of this section. 

\bt\label{convex-main}
Given a definably generated abelian connected group~$G$ in an o-minimal expansion $M$ of an ordered field, the following are equivalent: 
\begin{enumerate}
\item\label{convex-main-one}
For every definable set $X\subseteq G$, there is a definable set $Y \subseteq G$
such that $\repsum n X \subseteq \mult n Y$ for all~$n\in\N$.
\item\label{convex-main-two}
The group~$G$ is a cover of a definable group.
\item\label{convex-main-three}
The group~$G$ is divisible and has definably bounded convex hulls.
\end{enumerate}
Moreover these properties are stable under elementary extensions (i.e.\,preserved upwards and downwards). 
\et
\def\pnt#1{{\rm(\ref{convex-main-#1})}}

Theorem \ref{convex-main} says in particular that if a locally definable connected abelian group $G$ is generated by a definable convex subset, then $G$ is a cover of a definable group. Indeed it suffices that $G$ is generated by a set whose convex hull is contained in a definable set. 

\br \label{con} Condition (3) is equivalent to the conjunction of (1) and the condition that the torsion subgroup of $G$ is contained in a definable set. \er 
\bp Let us first observe that the convex hull of the identity of the group is the torsion subgroup of $G$. Therefore if (3) holds the torsion subgroup is contained in a definable set. Unraveling the definitions, (1) says that given a definable set $X \subseteq G$ there is a definable set $Y \subseteq G$, such that, for every positive integer $k$ and elements $a_1, \ldots, a_k\in X$ (not necessarily distinct), the equation $a_1 + \ldots + a_k = ky$ has {\em at least} one solution $y$ in $Y$. Condition (3) says almost the same thing, except that it is required that {\em all} the solutions $y$ of the above equation belong to $Y$ (it also says that $G$ is divisible, so there is at least one solution). Thus clearly (3) implies (1). 

Conversely assume that (1) holds and the torsion subgroup is contained in a definable set $Z \subseteq G$. Given a particular solution $y=y_0$ of $a_1 + \ldots + a_k = ky$, all the other solutions differ from $y_0$ by a $k$-torsion element, namely $y-y_0 \in G[k] \subseteq Z$. So if $X$ is given and we take $Y$ as granted by (1), then $Y + Z$ contains the convex hull of $X$ and we get (3).  
\ep 

In the proof of the Theorem we will make use of the following remark, which justifies the technical condition (1).  

\br \label{con2} Condition (1) in Theorem~\ref{convex-main} is inherited by the quotients of $G$, namely if $G$ satisfies (1) and $H<G$ is compatible, also $G/H$ satisfies (1). \er 

We will first prove Theorem \ref{convex-main} under the assumption that the o-minimal structure $M$ is sufficiently saturated. In the next section we show how to relax the saturation assumption.  
The next lemma contains the main idea of our argument. 

\begin{lemma}\label{convex-alias} (Assume $M$ $\aleph_1$-saturated.) 
Let $G$ be a locally definable abelian group generated by a definable
set~$X$ such that $0\in X$ and~$X=-X$. Assume that for some definable subset~$Y$ of~$G$
and for all positive $n\in\mathbb{N}$
we have~$\mult n Y \nsubseteq \repsum n X$.
Then there is an infinite cyclic compatible subgroup of~$G$.
\end{lemma}
\begin{proof}
By our assumptions for every $n$ there is an element $a_n \in Y$ with $na_n \nin \repsum n X$. Let $A \subseteq M$ be a countable set containing all the parameters needed to locally define $G$ and to define $X$ and $Y$. 
Let $p_n(x)$ be the type of $a_n$ over $A$. So $p_n(x) \in S_Y(A)$, where $S_Y(A)$ is the space of types in the parameters $A$ containing a defining formula for $Y$. Since $S_Y(A)$ is compact, there is some $q(x) \in S_Y(A)$ which is an accumulation point of $\{p_n: n > 0\}$. Let $b \in Y$ be a realization of the type $q$ (here we need $\aleph_1$-saturation). It suffices to show that $\Z b$ is an infinite compatible subgroup of $G$. We will actually prove that $\Z 2b$ is an infinite compatible subgroup of $G$, which is clearly equivalent (and in any case it suffices for our purposes). Choose $k$ so big that $Y \subseteq \Sigma_k X$ and let $n,m$ be any positive integers with $n \geq 4km$.

\noindent {\bf Claim.} For all $y\in Y$, if $ny \not\in \Sigma_n X$, then $2m y \not\in \Sigma_m X$.

Granted the claim, we have in particular $2m a_n \not\in \Sigma_m X$ for all $n \geq 4km$. Since the type $q(x)$ of $b$ is an accumulation point of $\{p_n: n> 0\}$, where $p_n$ is the type of $a_n$, it follows that  $2m b \not\in \Sigma_m X$ for all $m$. Thus for $m\to \infty$ the point $2mb$ eventually escapes from every definable subset of $G$. This implies that $\Z 2b$ is an infinite compatible subgroup of $G$. 
 
It remains to prove the claim. Let $n \geq 4km$ and let $y\in Y$ be such that $2my \in \Sigma_m X$.  We must show that $ny \in \Sigma_n X$. 
Write $n = 2mq + r$ with $r < 2m$. So $ny = 2m q y + ry \in \Sigma_{mq}X + \Sigma_{kr}X$. To prove $ny \in \Sigma_n X$ it suffices to show that $mq + kr \leq n$. Indeed we have $mq + kr < m \lfloor \frac n {2m} \rfloor + k 2m \leq n$ where the last inequality follows from the assumption $n \geq 4km$. The proof is thus complete.  
\end{proof}

\bc \label{cyclic-subgroup} (Assume $M$ $\aleph_1$-saturated.) 
Let $G$ be a locally definable abelian group generated by a definable set.
If $G$ is divisible and has property~\pnt{one} of Theorem~\ref{convex-main},
then either $G$ is definable or $G$ has an infinite cyclic compatible subgroup.
\ec
\begin{proof}
Let $X$ be a definable set generating~$G$. Without loss of
generality~$0\in X$ and~$X=-X$. Assume that $G$ is not definable, i.e.\ $n<m$ implies~ $\repsum m X
\nsubseteq \repsum n X$.
Let $Y'$ be~$\mult 2 Y$ where $Y$ is a definable set
witnessing property~\pnt{one} for~$X$. We claim that
for our sets~$X$ and~$Y'$ the hypothesis of Lemma~\ref{convex-alias} holds. In
fact $\mult n Y' = \mult {2n} Y \supseteq \repsum {2n} X \nsubseteq \repsum n X$.
\end{proof}

We are now ready to complete the proof of the Theorem in the saturated case. 

\begin{proof}[Proof of Theorem~\ref{convex-main} (for $M$ sufficiently saturated)] 
We have already remarked that \pnt{three} implies \pnt{one} (Remark \ref{con}). 

Now assume (1) with the aim of proving (2). Assume $M$ $\aleph_1$-saturated. We need to prove that 
$G$ covers a definable group.
Let $\Gamma$ be a torsion free discrete compatible subgroup of~$G$ of
maximal rank, which exists by Corollary~\ref{conj-b}. Consider~$G/\Gamma$. Observe that $G/\Gamma$ inherits property~\pnt{one} from~$G$ by Remark \ref{con2}. We claim that $G/\Gamma$ is definable. 
If not, then by Corollary~\ref{cyclic-subgroup} it has an infinite cyclic compatible
subgroup~$\Lambda$ and, as in the proof
of~\cite[Theorem~2.5]{Eleftheriou2012b}, the inverse image of~$\Lambda$ in~$G$ has rank
greater than~$\rank(\Gamma)$. This contradiction establishes \pnt{two}. 

Now assume \pnt{two} with the aim of proving \pnt{three}. Assume $M$ sufficiently saturated (as in Definition \ref{bounded}).  
By~\cite[Theorem~3.9]{Eleftheriou2012a} $G^{00}$ exists. By~\cite[Proposition~3.5]{Eleftheriou2012a}
$G^{00}$ is divisible and $G/G^{00}$, endowed with the logic topology, is an abelian connected real Lie~group. Let $X$ be an 
$M$-definable subset of $G$. We must find an $M$-definable set $D\subseteq G$ containing the convex hull of $X$. To this aim consider the natural projection  $\pi: G \to G/G^{00}$. 
We will make use of the following facts, which follow easily from the definition of the logic topology (see \cite{Eleftheriou2012a}):
\begin{enumerate}
\item[(i)] The image under $\pi$ of a definable subset of $G$ is a compact subset of $G/G^{00}$. 
\item[(ii)] The preimage under $\pi$ of a compact subset of $G/G^{00}$ is contained in a definable subset of $G$. 
\end{enumerate}
Moreover, since $G^{00}$ is divisible, the preimage under $\pi$ of a convex (in the sense of Definition \ref{convex}) subset of $G/G^{00}$ is easily seen to be a  convex (but not necessarily definable) subset of $G$. The strategy of the proof should now be clear. Given $X$, we consider its projection $\pi(X) \subseteq G/G^{00}$, which is compact. 
By \cite{Eleftheriou2012a}  $G/G^{00} \cong \R^k \times (\R/\Z)^l$ for some $k,l\in \N$. It then easily follows that any compact subset of $G/G^{00}$ is contained in a compact convex set. So $\pi(X)$ is contained in a compact convex set, and its preimage $\pi^{-1}(\pi(X))$ is convex and contained in a definable set $D$. Finally note that $D$ contains the convex hull of $X$. 
\end{proof}

We have thus completed the proof of Theorem \ref{convex-main} in the saturated case. In the next section we will show how to prove it in general. Let us however first derive some corollaries of the theorem.  

\bc \label{quot} Let $G$ be a locally definable abelian connected group. If $G$ covers a definable group, then also any quotient of $G$ by a compatible subgroup covers a definable group. \ec 
\bp By the equivalence between condition (1) and (2) in Theorem \ref{convex-main} and the fact that (1) is preserved under quotients. \ep

\begin{discussion}
Let us now discuss the conjecture of Eleftheriou and Peterzil (Conjecture \ref{cong-cover}) in the light of Theorem \ref{convex-main}. According to the conjecture, every definably generated connected abelian group is a cover a definable group. By Theorem \ref{convex-main} this is equivalent to the conjecture that every definably generated connected abelian group $G$ has definably bounded convex hulls. By Corollary \ref{quot}, passing to the universal cover, it suffices to state this last conjecture under the additional assumption that $G$ is simply connected. This may shed some light on the conjecture of Eleftheriou and Peterzil, at least if we assume divisibility. In fact the universal cover of a divisible locally definable abelian group is always divisible and torsion free (see Proposition \ref{compatible} and Proposition \ref{uc-torsion-free} below), and for divisible torsion free groups the convexity condition takes a very simple form (see Remark \ref{simple}).  So finally we are lead to the conjecture that a definably generated divisible torsion free group has definably bounded convex hulls. This is equivalent to the conjecture of Eleftheriou and Peterzil if we assume that the relevant groups are divisible.
\end{discussion}

We finish the section with a proof of the two propositions mentioned in the above discussion. 

\bprop\label{compatible}
Let $\pi\colon G \to H$ be a locally definable covering homomorphism
between locally definable connected abelian groups. If $H$ is divisible, then $G$
is divisible.
\eprop
\bp 
We must show that $nG= G$ for all~$n\in \N$. 
Without loss of generality we can assume $M$ sufficiently saturated (if not go to a saturated extension and note that the property to be proved is preserved). Since $H$ is divisible we
know that~$\pi(nG) = nH = H$. It follows that~$(nG)(\ker \pi) = G$. So $nG$ has
bounded index in~$G$ (see Definition \ref{bounded}). A locally definable subgroup of bounded index is
always a compatible subgroup~\cite[Fact~2.3]{Eleftheriou2012a}. Therefore $nG$ is a
compatible subgroup of~$G$. But since $nG$ is open and $G$ is connected,
it must coincide with~$G$.
\ep

\begin{proposition}\label{uc-torsion-free}
The universal cover of a locally definable abelian connected
divisible group is divisible and torsion free.
\end{proposition}
\begin{proof}
Let $G$ be our locally definable abelian connected divisible
group, and let $U$ be its universal cover. 
Using Proposition~\ref{compatible} we have the divisibility of~$U$, so for every positive $n\in \N$ multiplication
by $n$ is a covering homomorphism $n: U \to U$. Its induced homomorphism on the fundamental group is again given by multiplication by $n$, so by the results on coverings in Section \ref{covers} we have $U[n] \cong \pi_1(U)/n\pi_1(U)$. Since $\pi_1(U)=0$, we get $U[n]=0$. So $U$ is torsion free.  
\end{proof}

\section{Avoiding saturation}

In this section we prove Theorem \ref{convex-main} when $M$ is not assumed to be saturated. So in particular the theorem is valid when $M$ an o-minimal expansion of the field of real numbers (in which case $G(M)$ is a real Lie group). We need the following lemmas. 

\begin{lemma}\label{saturation}
Let $M'$ be an elementary extension of~$M$. Let $G$ be an~$M$-locally definable connected abelian group. If $G(M')$ covers an $M'$-definable group, then $G(M)$ covers some $M$-definable group (the converse is also true and obvious). 
\end{lemma}
A natural approach to the proof would be to use the characterization given in \cite{Eleftheriou2012a}: $G$ covers a definable group if and only if $G$ has a definable generic set. This latter condition is easily seen to transfer from $M'$ to $M$ and viceversa, but unfortunately the proof of this equivalence given in \cite{Eleftheriou2012a} uses saturation (via the consideration of $G/G^{00}$). So we need to analyze more closely the details of the argument in \cite{Eleftheriou2012a}. 
\begin{proof}[Proof of Lemma \ref{saturation}]
In \cite{Eleftheriou2012a} Eleftheriou and Peterzil proved that $G$ covers of  a definable group if and only if $G$ has a compatible subgroup $\Gamma$ such that $G/\Gamma$ is definable and $\Gamma$ is isomorphic to $\Z^k$ for some $k$. One direction is clear: if $\Gamma$ exists, $G$ covers the definable group $G/\Gamma$. The proof of the other direction in \cite{Eleftheriou2012a} however requires to work in a saturated structure $M$, since one makes use of $G/G^{00}$ in order to find the appropriate $\Gamma<G$. It turns out however that a modification of the proof 
in \cite{Eleftheriou2012a} gives the desired result. 
So assume that $G(M')$ covers an $M'$-definable group. We can assume $M'$ sufficiently saturated. As in \cite{Eleftheriou2012a}, working in $M'$, we have that  $G^{00}$ exists, and $G/G^{00}$ is an abelian Lie group isomorphic to $\R^k \times {(\R/\Z)}^l$ for some $k,l\in \N$. 
So we can write $G/G^{00}$ as the direct sum $L+K$ of two subgroups with $L\cong \R^k$ and $K\cong (\R/\Z)^l$. Note that $K$ is uniquely determined (it is the closure of the torsion subgroup), but $L$ is not. Our proof will rely on the possibility of making a suitable choice of $L$. To this aim fix a subgroup $\Gamma$ of $G/G^{00}$ isomorphic to $\Z^k$, say $\Z z_1 + \ldots + \Z z_k < G/G^{00}$, with $z_1, \ldots, z_k \in L$ (so $\Gamma < L$). Choose $u_1, \ldots, u_k\in G(M')$ such that $\pi(u_i) = z_i$, where $\pi: G\to G/G^{00}$ is the projection, and let $\Gamma = \Z u_1 + \ldots + \Z u_k < G(M')$. Eleftheriou and Peterzil show that $G(M')/\Gamma$ is $M'$-definable. If we could choose  $u_1, \dots, u_k$ to be in $G(M)$, the same argument would show that $G(M)/\Gamma$ is $M$-definable and we would be done. However it may happen that there are no points $u_i$ in $G(M)$ mapping to the given points $z_i \in G/G^{00}$. To overcome the impasse, we pick a
small open neighbourhood~$V_i\subseteq G/G^{00}$ of each~$z_i$. The
inverse image~$U_i\subseteq G$ of~$V_i$ is a union of subsets of~$G$ definable over the ground model $M$. Hence $U_i$
contains a non-empty $M$-definable set~$D_i$, and we pick each~$u_i$ in $D_i(M)$. Let $z'_i=\pi(u_i)\in V_i$. It is easy to see
that, if the neighbourhoods~$V_i$ were chosen small enough, there is an
automorphism $\psi$ of $G/G^{00}$ mapping each~$z'_i$ to the
corresponding~$z_i$. In fact, considering the universal cover $\R^{k+l}$ of $\R^k \times (\R/\Z)^l$ this amount to show that if we have a $k$-tuple of points in $\R^{k+l}$ whose $\R$-linear span is a $k$-dimensional subspace transversal to the subspace $0\times \R^l$, then the same remains true after a small perturbation of the points. We have thus reduced to the case when $\Gamma = \Z u_1 + \ldots + \Z u_k$ where $u_1, \ldots, u_k$ are in $G(M)$. 
To complete the proof that $G(M)/\Gamma$ is definable we must find a definable set containing representatives for all the cosets of $\Gamma$ in $G(M)$. For this we can reason as in~\cite[Lemma~3.3]{Eleftheriou2012a}, going to a saturated model $M'$ if needed, but then observing that every $M'$-definable set is contained in an $M$-definable set. 
\end{proof}

\bl \label{transfer} Assume that $G$ is locally definable over $M$ and let $M'$ be an elementary extension of $M$. Then the convexity condition (1) of Theorem \ref{convex-main} holds for $G(M)$ if and only if it holds for $G(M')$. Similarly for condition (3). \el
\bp The key observation is that, since $G$ is locally definable over $M$, every $M'$-definable subset $X$ of $G(M')$ is contained in an $M$-definable set $X_1 \subseteq G(M')$. Now observe that 
condition (1) has the following form: for every definable set $X\subseteq G$ there exists a definable set $Y\subseteq G$ satisfying, for all positive integer $k$, a suitable first order property $\phi_{k,X,Y}$, where $\phi_{k,X,Y}$ says that every equations of the form $a_1+\ldots + a_k = ky$, with $a_1, \ldots, a_k$ in $X$, has at least one solution $y$ in $Y$. Clearly if $X,Y$ are defined over $M$, then $\phi_{k,X,Y}$ holds in $M$ if and only if it holds in $M'$. Moreover by enlarging $X$ or $Y$ we can always assume, in checking (1), that $X$ and $Y$ are defined over $M$. So we have the result for (1), and the proof for (3) is similar.  
\ep 

\begin{proof}[Proof of Theorem \ref{convex-main} (general case)]
Let $G$ be locally definable over $M$ and let $M'$ be a sufficiently saturated elementary extension of $M$. We have already proved that Theorem \ref{convex-main} holds for $G(M')$. By Lemmas \ref{saturation} and \ref{transfer} we deduce it for $G(M)$. \end{proof}

\section{Further consequences of convexity}
\label{fconv} 

In this section we will show that the convexity hypothesis considered in Section~\ref{sec-convexity} can be used to study locally definable connected abelian groups even without assuming that the group is definably generated. 
In particular, we will prove that under suitable hypothesis the fundamental group $\pi_1(G)$ is finitely generated. We recall that every definable set $X$ has a finitely generated fundamental group 
\cite{Berarducci2002a}, but this result does not extend to locally definable sets. So we need to make essential use of the group structure of $G$. 

\begin{proposition}\label{discrete} Let $s \in \N$ and let $\Gamma$ be a subgroup of $\Q^s$ of rank $s$.
Let $v_1, \dotsc, v_s$ be $\Q$-independent elements of~$\Gamma$, and assume
that $\Gamma \cap \rh(v_1, \dotsc, v_s)$ is finite, where $\rh(v_1, \ldots, v_s)$ is the convex hull of $\{v_1, \ldots, v_s\}$. Then $\Gamma$ is isomorphic to~$\Z^s$. 
\end{proposition}
\bp
The hypothesis implies that the group identity is an isolated point of~$\Gamma$
in the topology inherited by~$\Q^d$ as a topological subgroup of~$\R^d$.
Now it suffices to recall that the only discrete subgroups of~$\R^d$ are
of form $\Z^s$ for some~$s$. 
\ep

\bt \label{fingen}
Let $G$ be a locally definable abelian connected group.
Assume that the universal cover~$U$ of~$G$ is divisible and has definably bounded convex
hulls.
Then $\pi_1(G) \cong \Z^s$ for some~$s\leq \dim(G)$. 
\et
\bp By Proposition \ref{uc-torsion-free} $U$ is divisible and torsion free, so it is a vector space over $\Q$. By the theory of covers $\pi_1(G)$ is isomorphic to a zero-dimensional compatible subgroup $\Gamma$ of $U$. So in particular $\pi_1(G)$ is torsion free. Moreover it has rank $\leq \dim(G)$ by Theorem \ref{pi-one}. It follows that $\Gamma$ is isomorphic to a subgroup of $\Q^s$ where $s$ is the rank of $\pi_1(G)$. 
Let $v_1,\dotsc,v_s\in \Gamma$ be $\Q$-linearly independent elements.
By the hypothesis
$\rh(v_1, \ldots, v_s)$ is contained in a definable subset~$D$ of~$U$.
Since $\Gamma$ is compatible and zero-dimensional, $\Gamma \cap D$ is finite.
So \textit{a fortiori} $\Gamma \cap \rh(v_1, \ldots, v_s)$ is finite.
Hence, by Proposition~\ref{discrete}, $\Gamma$ is isomorphic to~$\Z^s$.
\ep 

\bc \label{same} Under the same assumptions $G[n] \cong ({\Z/n\Z})^s$.
\ec 
\bp The assumption implies that $G$ is a connected divisible abelian locally definable group. 
Thus the multiplication by~$n$ is a covering homomorphism~$n\colon G\to
G$, and from the theory of covers $G[n] \cong
\pi_1(G)/n\pi_1(G)$~\cite[Theorem~3.15]{Edmundo2005}. Now apply theorem \ref{fingen}. 
\ep 

\bc \label{fgdisc} Let $G$ be a locally definable, abelian, connected group. 
Suppose that the universal cover of $G$ is divisible and has definably bounded convex hulls. 
Then every zero-dimensional compatible subgroup $\Gamma$ of $G$ is finitely generated, with at most $\dim(G)$ generators. 
\ec
\bp From the theory of covers it follows that $\Gamma$ is a quotient of $\pi_1(G/\Gamma)$, 
which is isomorphic to $\Z^s$ for some $s\leq \dim (G)$ by Theorem \ref{fingen}. 
\ep

\section{Examples and questions}
\label{examples}

\def\N{\mathbb{N}}
\def\st{\;|\;}

Let us recall the conjecture of ~\cite{Eleftheriou2012a}: 
every definably generated abelian connected group is a cover of a definable
group. We show that this conjecture cannot be generalized to non-abelian groups.

\begin{example}
There is a non-abelian locally definable group~$G$ such that $G$ is
generated by a definable generic set, but $G$ does not cover a definable group.
\end{example}
\bp 
Let $H$ be any definably connected centerless group of positive dimension
definable in an $\omega$-saturated real closed field~$M$. We will show that no
locally definable connected proper subgroup $G$ of~$H$ of maximal
dimension can cover a definable group; then reasoning as in \cite[Proposition 7.8]{Hrushovski2008}
we will construct such a subgroup which is generated by a definable generic set.
For the first part, let $G$ be a connected locally definable
proper subgroup of~$H$ with $\operatorname{dim}(G)=\operatorname{dim}(H)$.
Suppose that $G$ covers a definable group~$L$. Take any non-trivial
element~$x$ of~$G$ in the fibre over the identity of~$L$. Clearly $x$
must belong to the center of~$G$, hence the centralizer~$C_H(x)$ of~$x$
in~$H$ has maximal dimension. By the connectedness of~$H$ we
get~$C_H(x)=H$, hence $x\in Z(H)$, which is impossible since $H$
has trivial center.
Now we construct a locally definable subgroup~$G$ as before.
$H$ has a $C^1$ group-manifold structure. Taking a
local chart of~$H$ around the identity~$e$, we
can assume $e \in U\subseteq H$ for some open definable subset~$U$
of~$M^{\operatorname{dim}(G)}$ such that on $U$ the
group topology coincides with the subset topology. Without loss of
generality $e$ is $0\in{M}^{\operatorname{dim}(G)}$. Take a definable open
neighbourhood~$V$ of~$e$ such that the group operation restricted to
$V\times V$ is a differentiable function taking values in~$U$. Hence, for
all $x$ and~$y$ in~$V$, we have $x\cdot y = x + y + f(x,y)$ with
$f(x,y) \in o(|x|+|y|)$ for $(x,y)\to(e,e)$. By saturation, take an~$\epsilon>0$ such that
$|f(x,y)| \ll |x|+|y|$ for all $x$ and~$y$
with $|x|+|y|<\epsilon$ (where $a \ll b$ means $na < b$ for all $n\in \N$).
Basically, for elements smaller that $\epsilon$, the group structure
of~$H$ and the group structure of~$({M}^n, +)$ are infinitesimally
close, this enables us to construct our example in a straightforward way.
Pick a positive~$\delta\ll\epsilon$ and let $X$
be~$\{x\;|\;\left|x\right|<\delta\}$. We claim that the group~$G$ generated
by~$X$ is~$\bigcup_{i\in\mathbb{N}}X_i$, where
$X_i=\{x\;|\;\left|x\right|<i\delta\}$. In fact, clearly $X_i\cdot X_j
\subseteq X_{i+j+1}$, and for any~$x$ in~$X$ and
any~$n$ in~$\mathbb{N}$, we have $|(nx)\cdot x^{-n}| \ll |x|$, hence
$(nx)\cdot x^{-n}\in X$ and $X_n\subseteq X^{n+1}$. To prove that $X$ is generic
in~$G$ observe that there is a countable set~$Y=\{y_i\}_{i\in\mathbb{N}}$
such that $G=Y+\frac{1}{2}X$. It is easy to see that $G=Y\cdot X$.
\ep

With the next example we show that the fundamental group of a locally
definable H-space may not be finitely generated. (The reader may consult \cite{Dold1995} for the definition of H-space, and figure out the obvious adaptation of the definition to the locally definable category.) 
This may explain the difficulties in proving, without convexity assumptions, 
that the fundamental group of a locally definable connected abelian group is finitely generated. 

\begin{example}
There is a locally definable set $C$ with $\pi_1(C) \cong \Q$ and such that $C$ can 
be endowed with a locally definable H-space structure.
\end{example}
\bp 
Our H-space is going to be ``definably compact'' in the sense that all definable paths in it
have a limit.
Moreover we believe that the construction can be modified to yield a
locally definable manifold as well, however the details are very tedious
to verify.
Notice that the fundamental group of a compact topological manifold can
not be~$\mathbb{Q}$ \cite{Shelah1988}.
This is the plan of the example: working in a
real closed field~${M}$, first we describe a
definably compact locally definable space with $\mathbb{Q}$ as fundamental
group, then we show how to endow it with an H-space structure and how to embed it in ${M}^5$, and finally we suggest how to turn it into a locally definable manifold.
Let $\{G_i\}_{i=1,2,\dotsc}$ be infinitely many copies of~$SO(2)$. It is
useful to consider them as distinct groups. Consider the maps $m_i\colon
G_i\to G_{i+1}$ where $m_i(x) = ix$. Let $C_i$ be the mapping cylinder
of~$m_i$. By a slight abuse of notation we can consider $G_i$
and~$G_{i+1}$ as subspaces of~$C_i$. Then the definable fundamental group
of~$C=\bigcup_i C_i$ is~$\mathbb{Q}$. In fact, let $S_n$ denote the
initial segment~$\bigcup_{i\leq n} C_i$ of the union. Clearly $S_n$
retracts onto~$G_{n+1}$, hence $\pi_1(S_n) = \mathbb{Z}$. Also $\pi_1(C) =
\varinjlim \pi_1(S_n)$ where the maps are those induced by the inclusion.
Now, by the retraction, the inclusion of~$S_n$ in~$S_{n+1}$ induces on the
fundamental groups
the same map as the inclusion of~$G_{n+1}$ in~$C_{n+1}$, which is the
multiplication by~$n+1$. So $\pi_1(C)$ is the group generated by the
generators~$\alpha_i$ of~$\pi_1(G_i)$ with
relations~$i\alpha_i=\alpha_{i+1}$, and this is~$\mathbb{Q}$.
The H-space structure is given by the following map. Let $\mathbf{x}=(x,t)\in
C_i\setminus G_{i+1}\simeq G_i\times [0,1)$ and~$\mathbf{y}=(y,s)\in C_j\setminus
G_{j+1}\simeq G_j\times [0,1)$ be two elements of~$C$. If $i=j$ then we
let $\mu(\mathbf{x},\mathbf{y}) = (x\cdot y,\max(t,s))\in C_i\setminus
G_{i+1}$. If $i<j$ we let $\mu(\mathbf{x},\mathbf{y}) = (\rho_j(x)\cdot
y,s)\in C_j\setminus G_{j+1}$ where $\rho_j$ is the retraction from
$C_1\cup\dotsb\cup C_{j-1}$ to~$G_j$. The case~$j<i$ is symmetric.
For the embedding in~${M}^5$, we first embed $SO(2)$
in~${M}^2$ in the standard way. Subsequently we map $G_i$ to
$SO(2)\times\{(0,0,i)\}\subseteq{M}^5$. The mapping cylinder~$C_i$
maps to
\[
\left\{\left(\left(1-t\right)a+ta^n,t\left(1-t\right)a,i+t\right)\;|\;a\in SO(2), t\in
[0,1]\right\}\subseteq
{M}^2\times{M}^2\times{M}={M}^5
\]
$C$ is the union of the mapping cylinders. Finally,
our H-space is not a definable manifold. We suggest that taking a suitably
small tubular neighbourhood of it (in~${M}^5$) should yield a
homotopy equivalent locally definable manifold.
\ep

The following questions remain open: 

\begin{question} \label{ques} Assume $G$ is a locally definable abelian connected group. 
\begin{enumerate}
\item Is $G$ divisible? \cite{Edmundo2003a,Edmundo2005,Eleftheriou2012a}
\item Suppose $G$ is definably generated. Is $G$ a cover of a definable group? \cite{Eleftheriou2012a}
\item Is the torsion subgroup of $G$ contained in a definable set?
\item Is $\pi_1(G)$ finitely generated? 
\end{enumerate}
\end{question}

\bibliographystyle{amsalpha}

%\bibliography{BEM}
\providecommand{\bysame}{\leavevmode\hbox to3em{\hrulefill}\thinspace}
\providecommand{\MR}{\relax\ifhmode\unskip\space\fi MR }
\providecommand{\MRhref}[2]{%
  \href{http://www.ams.org/mathscinet-getitem?mr=#1}{#2}
}
\providecommand{\href}[2]{#2}

\end{document}